\documentclass[11pt,bezier]{article}
\usepackage{amsmath,amssymb,amsfonts,euscript,graphicx,algorithm,algorithmic,makeidx,amscd,enumerate,setspace,tkz-berge}

\newtheorem{prethm}{{\bf Theorem}}

\newenvironment{thm}{\begin{prethm}{\hspace{-0.5
               em}{\bf.}}}{\end{prethm}}

\newtheorem{prepro}[prethm]{{\bf Theorem}}

\newtheorem{preprop}[prethm]{{\bf Proposition}}

\newenvironment{prop}{\begin{preprop}{\hspace{-0.5
               em}{\bf.}}}{\end{preprop}}

\newtheorem{precor}[prethm]{{\bf Corollary}}

\newenvironment{cor}{\begin{precor}{\hspace{-0.5
               em}{\bf.}}}{\end{precor}}

\newtheorem{predefn}[prethm]{{\bf Definition}}

\newtheorem{preconj}[prethm]{{\bf Conjecture}}

\newtheorem{preremark}[prethm]{{\bf Remark}}

\newenvironment{remark}{\begin{preremark}\rm{\hspace{-0.5
               em}{\bf.}}}{\end{preremark}}

\newtheorem{preexample}[prethm]{{\bf Example}}

\newtheorem{prelem}[prethm]{{\bf Lemma}}

\newenvironment{lem}{\begin{prelem}{\hspace{-0.5
               em}{\bf.}}}{\end{prelem}}

\newtheorem{prelam}{{\bf Lemma}}

\newtheorem{preproof}{{\bf Proof.}}

\newenvironment{proof}[1]{\begin{preproof}{\rm
               #1}\hfill{$\Box$}}{\end{preproof}}

\title{\bf \large  THE  COLORING OF THE REGULAR GRAPH OF IDEALS
}

\author{{\normalsize {Farzad Shaveisi}
}\vspace{3mm}\\
{\footnotesize{\it Department of Mathematics, Faculty of Sciences, Razi University, Kermanshah, Iran }}\\
{\footnotesize{ $\mathsf{f.shaveisi@razi.ac.ir}$}}}
\date{}
\begin{document}
\maketitle
\vspace{-.4cm}
\noindent{\small{\sc Abstract}. The regular graph of ideals of the commutative ring $R$, denoted by $\Gamma_{reg}(R)$, is a graph whose vertex
set is the set of all non-trivial ideals of $R$ and two distinct vertices $I$ and $J$ are adjacent if and only if either $I$ contains a $J$-regular element or $J$ contains an $I$-regular element. In this paper, it is shown that for every Artinian ring $R$,  the  edge chromatic number of $\Gamma_{reg}(R)$ equals its maximum degree. Then a formula for the clique number of $\Gamma_{reg}(R)$ is given. Also, it is proved that
for every reduced ring $R$ with $n(\geq3)$ minimal prime ideals, the edge chromatic number of $\Gamma_{reg}(R)$ is $2^{n-1}-2$.  Moreover, we show that both of the clique number and vertex chromatic number of $\Gamma_{reg}(R)$ are $n-1$, for every reduced ring $R$ with $n$ minimal prime ideals.


\vspace{5mm} \noindent{\bf\large 1. Introduction}\\ 

We begin with recalling some  notations on graphs.   Let $\Gamma$ be a digraph.
We denote the vertex set of $\Gamma$, by $V(\Gamma)$. Also, we distinguish the {\it{out-degree}} $d_\Gamma^+(v)$, the number of arcs leaving a vertex $v$, and the {\it{in-degree}} $d_\Gamma^-(v)$, the number of arcs entering a vertex $v$. If the graph is oriented, the degree $d_\Gamma(v)$ of a vertex $v$ is equal to the sum of its out- and in-degrees.
Let $G$ be a simple graph with the vertex set $V(G)$ and $A\subseteq V(G)$.  We denote by
$G[A]$ the subgraph of $G$ induced by $A$.
If $|V(G)|=\mu$, for some cardinal number $\mu$, then the complete graph and its complement are denoted by $K_\mu$ and $\overline{K_\mu}$, respectively. The degree of a vertex $x$ of $G$ is denoted by $d(x)$ and the maximum degree of vertices of $G$ is denoted by $\Delta(G)$.
A complete bipartite graph with parts of sizes $\mu$ and $\nu$ is denoted by $K_{\mu,\nu}$. Moreover, if either $\mu=1$ or $\nu=1$, then the complete
bipartite graph is said to be a \textit{star graph}.
Let $G_1$ and $G_2$ be two arbitrary
graphs. By $G_1+G_2$ and  $G_1\vee G_2$, we mean the {\it disjoint union}
of $G_1$ and $G_2$ and {\it join} of two graphs $G_1$ and $G_2$,
respectively.
For a
graph $G$, the \textit{clique number} of $G$,
and  the {\textit{vertex} ({\it edge}) {\it chromatic number} of $G$ are denoted by $\omega(G)$, and $\chi(G)$ ($\chi'(G)$), respectively. For more details about the used
terminology of graphs, see \cite{west}.

Throughout this
paper, $R$ is assumed to be a non-domain commutative ring with identity.
 An element $r\in R$ is called
\textit{$R$-regular} if $r\notin Z(R)$, where $Z(R)$ denotes the set of all zero-divisors of $R$.
By $\mathbb{I}(R)$ ($\mathbb{I}(R)^*$), $Max(R)$ and $Min(R)$ we
denote the set of all proper (non-trivial) ideals of $R$, the set of all maximal ideals of $R$ and the set of all minimal prime ideals of $R$, respectively. The ring $R$
is said to be \textit{reduced}, if it has no non-zero nilpotent
element. For every ideal $I$ of $R$, the {\it annihilator of $I$} is denoted by $Ann(I)$. A subset $S$ of a commutative
ring $R$ is called a {\it{multiplicative closed subset {\rm (m.c.s)}}} of
$R$ if $1\in S$ and $x,y\in S$ implies that $xy\in S$. If $S$ is an
m.c.s of $R$ and $M$ is an $R$-module, then we denote by $R_S$ and
$M_S$, the ring of fractions of $R$ and the module of fractions
 of $M$ with respect to $S$, respectively. If ${\mathfrak{p}}$ is a prime ideal of $R$ and $S=R\setminus {\mathfrak{p}}$, we use the notation $M_{\mathfrak{p}}$, for the localization of $M$
  at ${\mathfrak{p}}$. By $T(R)$, we mean the {\it{total ring}} of $R$ that is the ring of fractions, where $S=R\setminus Z(R)$.

As we know, most properties of a ring are closely tied to the
behavior of its ideals, so it is useful to study graphs or
digraphs, associated to the ideals of a ring or associated to modules. To see an instance of these
graphs, the reader is referred to \cite{disc.math,rocky,atani,MBI,intersection,actamathhungar,redmond,Safaeeyan,F.Sarei}.
The
{\it regular digraph of ideals} of a ring $R$, denoted by
$\overrightarrow{\Gamma_{reg}}(R)$, is a digraph whose vertex set
is the set of all non-trivial ideals of $R$ and for every two
distinct vertices $I$ and $J$, there is an arc from $I$ to  $J$ if
and only if $I$ contains a $J$-regular element. The
underlying graph of $\overrightarrow{\Gamma_{reg}}(R)$ is denoted by $\Gamma_{reg}(R)$. The regular digraph (graph) of ideals, first was introduced by Nikmehr and Shaveisi in \cite{actamathhungar}.
Then in \cite{khashayarmanesh}, Afkhami, Karimi and Khashayarmanesh followed the study of this graph. In this paper, the coloring of the regular graph of ideals is studied. In Section 2, it is shown that $\chi'(\Gamma_{reg}(R))=\Delta(\Gamma_{reg}(R))$, where $R$ is an Artinian ring. In Section 3, it is shown that $\chi(\Gamma_{reg}(R))=\omega(\Gamma_{reg}(R))=2|{\rm Max}(R)|-f(R)-1$, where $R$ is an Artinian ring and $f(R)$ denotes the number of fields, appeared in the decomposition of $R$ to direct product of local rings. Section 4 is devoted to the case that $R$ is a reduced ring.
For example, for every reduced ring $R$ with $|{\rm Min}(R)|=n\geq 3$, we obtain that $\chi(\Gamma_{reg}(R))=\omega(\Gamma_{reg}(R))=n-1$
and $\chi'(\Gamma_{reg}(R))=2^{n-1}-2$.

\vspace{7mm} \noindent{\bf\large 2. The Edge Chromatic Number
}\\\\ 
In this section, we study the edge coloring of the regular graph of ideals of an Artinian ring. Before this, we need the following lemma from \cite{beineke}.
\begin{lem}\label{firstkindgraph}{\rm\cite[Corollary 5.4]{beineke}}
Let $G$ be a simple graph. Suppose that for every
vertex $u$ of maximum degree, there exists an edge $\{u,v\}$ such that $\Delta(G)-d(v)+2$ is
more than the number of vertices with maximum degree in $G$. Then $\chi'(G)=\Delta(G)$.
\end{lem}

%

\begin{remark}\label{remark}
Let $R_1,\ldots,R_n$ be rings, $R\cong R_1\times \cdots\times R_n$ and $I=I_1\times \cdots\times I_n$ and $J=J_1\times \cdots\times J_n$ be
two distinct vertices of ${\Gamma_{reg}}(R)$. Then

{\rm(i)}
 $I$ contains a $J$-regular element if and only if for every $i$, either $I_i$ contains a $J_i$-regular element or $J_i=(0)$.

{\rm(ii)}
Assume that  every $R_i$ is an Artinian local ring. Then (i) and  {\rm\cite[Theorem 2.1]{actamathhungar}} imply that if $I$ contains a $J$-regular element, then $J$ contains no $I$-regular element.
\end{remark}


\begin{thm}\label{edgechromaticnumber}
 If $R$ is an Artinian ring, then $\chi'(\Gamma_{reg}(R))=\Delta(\Gamma_{reg}(R))$.
\end{thm}
\begin{proof}
{Let $R$ be an Artinian ring. Then by \cite[Theorem 8.7]{ati}, there exists a positive integer $n$ such that $R\cong R_1\times\cdots\times R_n$, where every $R_i$ is an Artinian local ring. If $R$ contains infinitely many ideals, then with no loss of generality, we can assume that $\mathbb{I}(R_1)$ is an infinite set. Since $(0)\times R_2\times\cdots\times R_n$ is adjacent to $I_1\times (0)$, for every non-zero ideal $I_1$ of $R_1$, we deduce that, $\chi'(\Gamma_{reg}(R))=\Delta(\Gamma_{reg}(R))=\infty$. Therefore, one can suppose that $|\mathbb{I}(R)^*|<\infty$. If $R$ is a local ring, then by \cite[Theorem 2.1]{actamathhungar},  $\Gamma_{reg}(R)\cong \overline{K_{|\mathbb{I}(R)^*|}}$ and hence $\chi'(\Gamma_{reg}(R))=\Delta(\Gamma_{reg}(R))=0$. For the non-local case, we continue the proof in the following three cases:

Case 1. $R$ is a reduced ring. Since $R$ is Artinian, we conclude that $R\cong F_1\times\cdots\times F_n$, where every $F_i$ is a field.
If $n\leq 5$, then it is not hard to check that $\chi'(\Gamma_{reg}(R))=\Delta(\Gamma_{reg}(R))=2^{n-1}-2$. Thus we can suppose that $n\geq 6$. Now, let $I=F_1\times\cdots\times F_r\times (0)\times\cdots\times (0)$ be a vertex of $\Gamma_{reg}(R)$, where $1\leq r\leq n-1$. Then we have:
$$d(I)=d^+(I)+d^-(I)=2^{r}-2+2^{n-r}-2=2^{r}+2^{n-r}-4.$$
Therefore, a vertex $I=I_1\times\cdots\times I_n$ of $\Gamma_{reg}(R)$ has maximum degree if and only if either there exists exactly one $j$, $1\leq j\leq n$ such that  $I_j=F_j$ or there exists exactly one $j$, $1\leq j\leq n$ such that  $I_j=(0)$. So, the number of vertices with maximum degree is $2n$. Now, let $u$ be a vertex with maximum degree. Then with no loss of generality, we can suppose that either $u=F_1\times (0)\times\cdots\times(0)$ or $u=(0)\times F_1\times\cdots\times F_n$. Suppose that $u=F_1\times (0)\times\cdots\times(0)$ and consider the vertex $v=F_1\times\cdots\times F_{[\frac{n}{2}]}\times(0)\times\cdots\times (0)$. Clearly,
$d(u)=\Delta(\Gamma_{reg}(R))=2^{n-1}-2$, $d(v)=2^{[\frac{n}{2}]}+2^{n-[\frac{n}{2}]}-4$ and $u$ is adjacent to $v$.
Since $n\geq 6$, we deduce that
$$\Delta(\Gamma_{reg}(R))-d(v)+2=2^{n-1}-2^{n-[\frac{n}{2}]}-2^{[\frac{n}{2}]}+4>2n.$$
If $u=(0)\times F_1\times\cdots\times F_n$, then  a similar proof to that of above shows that $\Delta(\Gamma_{reg}(R))-d(v)+2>2n$. Thus Lemma \ref{firstkindgraph} implies that $\chi'(\Gamma_{reg}(R))=\Delta(\Gamma_{reg}(R))$.

Case 2. $R$ is a non-reduced ring and $|{\rm Max}(R)|=2$. In this case, $R\cong R_1\times R_2$, where $R_1$ and $R_2 $ are Artinian local rings. Let $|V(\Gamma_{reg}(R_1))|=\mu$ and $|V(\Gamma_{reg}(R_2))|=\nu$,  $$A=V(\Gamma_{reg}(R_1))\times V(\Gamma_{reg}(R_2)),$$
$$B_1=V(\Gamma_{reg}(R_1))\times\{(0)\},\, B_2=\{R_1\}\times V(\Gamma_{reg}(R_2)),\, B_3=\{R_1\times (0)\},$$
$$C_1=V(\Gamma_{reg}(R_1))\times\{R_2\},\, C_2=\{(0)\}\times V(\Gamma_{reg}(R_2)),\,C_3=\{(0)\times R_2\},$$
$$B=B_1\cup B_2\cup B_3\,\, {\rm and}\,\,C=C_1\cup C_2\cup C_3.$$
Then we have: $$\Gamma_{reg}(R)\cong \Gamma_{reg}(R)[A]+\Gamma_{reg}(R)[B]+\Gamma_{reg}(R)[C]\cong \overline{K_{\mu\nu}}+ (\overline{K_{\mu}}\vee K_{1,\nu})+ (\overline{K_{\mu}}\vee K_{1,\nu}).$$
So, $\chi'(\Gamma_{reg}(R))=\mu +\nu=\Delta(\Gamma_{reg}(R))$, as desired.

Case 3. $R$ is a non-reduced ring and $|{\rm Max}(R)|=n\geq 3$. Let $I=I_1\times\cdots\times I_n$ be a non-trivial ideal of $R$ and define the following sets and numbers:
$$\Delta_I=\{k\,|\,1\leq k\leq n\,\, {\rm and} \,\,I_k=R_k\};$$
$$\Upsilon_I=\{k\,|\,1\leq k\leq n\,\, {\rm and} \,\,I_k=(0)\};$$
$$\Lambda_I=\{k\,|\,1\leq k\leq n,\,\, {\rm and}\,I_k \,{\rm is \,a\,\,non-trivial\,ideal\,of} R_k\};$$
$$t_i=|\mathbb{I}(R_i)|;\ (1\leq i\leq n);$$
$$T_i=\{j|\ 1\leq j\leq n\ {\rm and} \ |\mathbb{I}(R_j)|=t_i\};\ s_i=|T_i|\ (1\leq i\leq n).$$
With no loss of generality, we can assume that $t_1\geq\cdots\geq t_n$. Now, let us compute the degree of every vertex of $\Gamma_{reg}(R)$. By Remark \ref{remark}, there is an arc from $I$ to $J$ in $\overrightarrow{\Gamma_{reg}}(R)$ if and only if  $\Upsilon_J\supseteq \Upsilon_I\cup\Lambda_I$. So, the out-degree of $I$ in $\overrightarrow{\Gamma_{reg}}(R)$ equals:
$$d^+(I)=
  \begin{cases}
   0;                                   & \Delta_I=\varnothing\\
   \prod_{k\in \Delta_I}(t_k+1)-2;      & \Delta_I\neq\varnothing\ {\rm and}\ \Lambda_I=\varnothing\,\,\, \\
   \prod_{k\in \Delta_I}(t_k+1)-1;       & \Delta_I\neq\varnothing\ {\rm and}\ \Lambda_I\neq\varnothing\,\,\, \,\,
 \end{cases}$$
Also, Remark \ref{remark} implies that there is an arc from $J$ to $I$ in $\overrightarrow{\Gamma_{reg}}(R)$ if and only if $\Delta_J \supseteq\Delta_I\cup\Lambda_I$. Thus the in-degree of $I$ in $\overrightarrow{\Gamma_{reg}}(R)$ equals:
$$d^-(I)=
  \begin{cases}
    0;                                & \Upsilon_I=\varnothing\\
    \prod_{k\in \Upsilon_I}(t_k+1)-2;      & \Upsilon_I\neq\varnothing \ {\rm and}\ \Lambda_I=\varnothing\,\,\, \\
    \prod_{k\in \Upsilon_I}(t_k+1)-1;       & \Upsilon_I\neq\varnothing \ {\rm and}\ \Lambda_I\neq\varnothing.\,\,\, \,\,
 \end{cases}$$
Therefore,
 $$d(I)= \begin{cases}
    0;                                   & \Lambda_I=\{1,\ldots,n\}\\
   \prod_{k\in \Delta_I}(t_k+1)+\prod_{k\in \Upsilon_I}(t_k+1)-4;      & \Lambda_I=\varnothing\,\,\, \\
   \prod_{k\in \Upsilon_I}(t_k+1)-1;       &\Lambda_I\neq\varnothing,\Upsilon_I\neq\varnothing\ {\rm and}\ \Delta_I=\varnothing\,\, \,\,\\
   \prod_{k\in \Delta_I}(t_k+1)-1;       &\Lambda_I\neq\varnothing,\Upsilon_I=\varnothing\ {\rm and}\ \Delta_I\neq\varnothing\,\, \,\,\\
   \prod_{k\in \Delta_I}(t_k+1)+\prod_{k\in \Upsilon_I}(t_k+1)-2; &\Lambda_I\neq\varnothing,\Upsilon_I\neq\varnothing\ {\rm and}\ \Delta_I\neq\varnothing.\,\, \,\,
 \end{cases}$$
\noindent Now, we consider the following two subcases:

Subcase 1. $R$ contains no field as its direct summand. From the above argument, we conclude that $I$ has maximum degree if and only if either $\Delta_I=\{1,\ldots,n\}\setminus\{j\}$ and $\Lambda_I=\{j\}$ or $\Upsilon_I=\{1,\ldots,n\}\setminus\{j\}$ and $\Lambda_I=\{j\}$, for some $j\in T_n$. Let $u$ be a vertex with maximum degree in $\Gamma_{reg}(R)$. Then with no loss of generality, we can suppose that either $u=R_1\times\cdots\times R_{n-1}\times I_n$ or $u=(0)\times\cdots\times (0)\times J_n$, where $I_n$ and $J_n$ are non-trivial ideals of $R_n$. First suppose that $u=R_1\times\cdots\times R_{n-1}\times I_n$, where $I_n$ is non-trivial ideal of $R_n$. Consider the vertex $v=\mathfrak{m}_1\times\cdots\times \mathfrak{m}_{n-1}\times (0)$, where $\mathfrak{m}_i$ is the maximal ideal of $R_i$, for every $1\leq i\leq n$. By Remark \ref{remark}, a vertex $J$ in $\Gamma_{reg}(R)$ is adjacent to $v$ if and only if $J=R_1\times\cdots\times R_{n-1}\times J_n$, for some proper ideal $J_n$ of $R_n$. Therefore, the following statements are true:
 \begin{enumerate}
 \item[\rm (a)] $u$ and $v$ are two adjacent vertices in $\Gamma_{reg}(R)$ and $d(v)=t_n$.
 \item[\rm (b)] $\Delta(\Gamma_{reg}(R))=d(u)=\prod_{k=1}^{n-1}(t_k+1)-1$.
 \item[\rm (c)] The number of vertices with maximum degree in $\Gamma_{reg}(R)$ is $2s_n(t_n-1)$.
 \end{enumerate} Since $n\geq 3$ and $t_i\geq 2$, for every $i\geq1$, from the above statements, we have:
 \begin{eqnarray*}\label{one}
\Delta(\Gamma_{reg}(R))-d(v)+2-2s_n(t_n-1)=\prod_{k=1}^{n-1}(t_k+1)-(2s_n+1)(t_n-1)\\
 \geq  3^{n-2}(t_n+1)-(2n+1)(t_n-1)
\\=(3^{n-2}-2n-1)+2(3^{n-2})>0
 \end{eqnarray*}
Hence  $\Delta(\Gamma_{reg}(R))-d(v)+2$ is more than the number of vertices with maximum degree in $\Gamma_{reg}(R)$. Thus by Lemma \ref{firstkindgraph}, $\chi'(\Gamma_{reg}(R))=\Delta(\Gamma_{reg}(R))$. Now, suppose that $u=(0)\times\cdots\times (0)\times J_n$, for some non-trivial ideal  $J_n$ of $R_n$. In this case, consider the vertex $v=\mathfrak{m}_1\times\cdots\times \mathfrak{m}_{n-1}\times R_n$, where $\mathfrak{m}_i$ is the maximal ideal of $R_i$, for every $1\leq i\leq n$. Then a similar argument to that of above shows that $u$ and $v$ are adjacent and $\Delta(\Gamma_{reg}(R))-d(v)+2$ is more than the number of vertices with maximum degree in $\Gamma_{reg}(R)$. So, Lemma \ref{firstkindgraph} implies that $\chi'(\Gamma_{reg}(R))=\Delta(\Gamma_{reg}(R))$.

Subcase 2. $R$ contains a field as its direct summand. In this case, $R_n$ is a field. From the argument before subcase 1, we conclude that $I$ has maximum degree if and only if either $\Delta_I=\{1,\ldots,n\}\setminus\{j\}$ and $\Upsilon_I=\{j\}$ or $\Upsilon_I=\{1,\ldots,n\}\setminus\{j\}$ and $\Delta_I=\{j\}$, for some $j\in T_n$. So, if $u$ is a vertex with maximum degree in $\Gamma_{reg}(R)$, then with no loss of generality, we can suppose that either $u=R_1\times\cdots\times R_{n-1}\times (0)$ or $u=(0)\times \cdots\times (0)\times R_n$. First suppose that $u=R_1\times\cdots\times R_{n-1}\times (0)$. Consider the vertex $v=\mathfrak{m}_1\times\cdots\times \mathfrak{m}_{n-k}\times(0)\times\cdots\times (0)$, where $k$ is the number of fields, appearing in the decomposition of $R$ to local rings. Then it is clear that a vertex $J=J_1\times\cdots\times J_n$ is adjacent to $v$ if and only if $J_i=R_i$, for every $1\leq i\leq n-k$. Therefore, the following statements are true:
\begin{enumerate}
\item[\rm(a$^{\prime}$)] $u$ and $v$ are two adjacent vertices in $\Gamma_{reg}(R)$ and $d(v)=2^k$.
 \item[\rm (b$^{\prime}$)] $\Delta(\Gamma_{reg}(R))=d(u)=\prod_{k=1}^{n-1}(t_k+1)-2$.
 \item[\rm (c$^{\prime}$)] The number of vertices with maximum degree in $\Gamma_{reg}(R)$ is $2k$.
  \end{enumerate}
Thus $\Delta(\Gamma_{reg}(R))-d(v)+2=\prod_{k=1}^{n-1}(t_k+1)-2^k$. Since $R$ is not reduced and $n\geq 3$, we deduce that this number is greater than the number of vertices with maximum degree in $\Gamma_{reg}(R)$ and hence Lemma \ref{firstkindgraph} implies that $\chi'(\Gamma_{reg}(R))=\Delta(\Gamma_{reg}(R))$. Now, assume that $u=(0)\times \cdots\times (0)\times R_n$ and consider the vertex $v=\mathfrak{m}_1\times\cdots\times \mathfrak{m}_{n-k}\times R_{k+1}\times\cdots\times R_n$. Then a similar proof to that of above shows that $u$ and $v$ are adjacent and we have:
\begin{eqnarray*}
\Delta(\Gamma_{reg}(R))-d(v)+2=\prod_{k=1}^{n-1}(t_k+1)-2^k+1\geq 3^{n-1}-2^k+1,
\end{eqnarray*}
which is greater than  the number of vertices with maximum degree in $\Gamma_{reg}(R)$. Thus by Lemma \ref{firstkindgraph}, $\chi'(\Gamma_{reg}(R))=\Delta(\Gamma_{reg}(R))$ and so the proof is complete.~}
\end{proof}

\vspace{5mm} \noindent{\bf\large 3. A Formula for the Clique Number in Artinian Rings}\\ 

Let $R=\mathbb{Z}_{36}\cong \mathbb{Z}_4\times \mathbb{Z}_9$. Then it is clear that $R$ is an Artinian ring with two maximal ideals. On the other hand, we know that $C=\{\mathbb{Z}_4\times (0), \mathbb{Z}_4\times (3), (2)\times (0)\}$ is a clique in $\Gamma_{reg}(R)$ and so $\omega(\Gamma_{reg}(R))\geq 3>|{\rm Max}(R)|$; therefore, this implies that the upper bound for $\omega(\Gamma_{reg}(R))$ in \cite[Theorem 2.1]{actamathhungar} is incorrect. The regular digraph of $\mathbb{Z}_{36}$ is seen in the following figure:
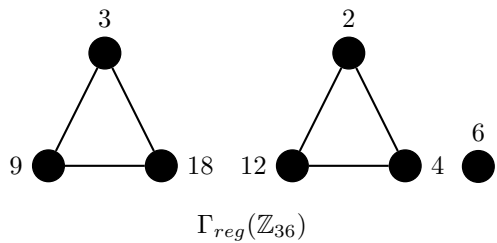
\begin{center}
\begin{tikzpicture}
        \GraphInit[vstyle=Classic]
        \Vertex[x=0,y=1.5,style={black,minimum size=3pt},LabelOut=true,Lpos=90]{3}
        \Vertex[x=0.75,y=0,style={black,minimum size=3pt},LabelOut=true,Lpos=0]{18}
        \Vertex[x=-0.75,y=0,style={black,minimum size=3pt},LabelOut=true,Lpos=180]{9}

        \Edges(3,18)
        \Edges(9,3)
        \Edges(9,18)
\end{tikzpicture}
\begin{tikzpicture}
        \GraphInit[vstyle=Classic]
        (\Vertex[x=0,y=1.5,style={black,minimum size=3pt},LabelOut=true,Lpos=90]{2})
        \Vertex[x=0.75,y=0,style={black,minimum size=3pt},LabelOut=true,Lpos=0]{4}
        \Vertex[x=-0.75,y=0,style={black,minimum size=3pt},LabelOut=true,Lpos=180]{12}

        \Edges(2,4)
        \Edges(2,12)
        \Edges(4,12)
\end{tikzpicture}
\begin{tikzpicture}
        \GraphInit[vstyle=Classic]
        \Vertex[x=-0.75,y=0,style={black,minimum size=3pt},LabelOut=true,Lpos=90]{$6$}
        \end{tikzpicture}

    
    $\Gamma_{reg}(\mathbb{Z}_{36})$
\end{center}

In this section, we give a correct upper bound for $\omega(\Gamma_{reg}(R))$, when $R$ is an Artinian ring.  In fact, it is shown that for every Artinian ring $R$, $|{\rm Max}(R)|-1\leq\omega(\Gamma_{reg}(R))\leq 2|{\rm Max}(R)|-1$ and the lower bound occurs if and only if $R$ is a reduced ring. If $R$ is an Artinian local ring which is not a field, then by {\rm\cite[Theorem 2.1]{actamathhungar}}, $\omega(\Gamma_{reg}(R))=1$. Also, it is clear that for every field $F$, $\omega(\Gamma_{reg}(F))=0$.

\begin{lem}\label{omegaproduct}
Let $S$ be an Artinian ring and $T$ be an Artinian local ring. If $R\cong S\times T$, then
$$\omega(\Gamma_{reg}(R))=
\begin{cases}
\omega(\Gamma_{reg}(S))+1; & T\, {\rm is \ a \ field}\\
\omega(\Gamma_{reg}(S))+2; & T\, {\rm is \ not \ a \ field}.\\
 \end{cases}
 $$
\end{lem}
\begin{proof}
{First note that for every clique $C$ of $\Gamma_{reg}(S)$,  $C\times \{T\}$ is a clique of $\Gamma_{reg}(R)$. Also, for any clique $C'=\{I_i\times J_i\}_{i\in A}$ of $\Gamma_{reg}(R)$, from Remark \ref{remark} and \cite[Theorem 2.1]{actamathhungar}, we deduce that $\{J_i|\ I_i\times J_i\in C'\}_{i\in A}$ contains at most one nontrivial ideal. Therefore, $\omega(\Gamma_{reg}(S))$  is infinite if and only if $\omega(\Gamma_{reg}(R))$ is infinite.
Now, assume that $\omega(\Gamma_{reg}(S))$ is finite and $C$ is a  clique of $\Gamma_{reg}(S)$  with  $|C|=\omega(\Gamma_{reg}(S))$. If $T$ is a field, then  $C\times\{T\}\cup\{(0)\times T\}$ is a clique of $\Gamma_{reg}(S)$. Also, if  $T$ is not a field, then for every nontrivial ideal $J$ of $T$, $C\times\{T\}\cup\{(0)\times J,(0)\times T\}$ is a clique of $\Gamma_{reg}(R)$. Therefore,
 $$\omega(\Gamma_{reg}(R))\geq
\begin{cases}
\omega(\Gamma_{reg}(S))+1; & T\, {\rm is \ a \ field}\\
\omega(\Gamma_{reg}(S))+2; & T\, {\rm is \ not \ a \ field}.\\
 \end{cases}
 $$
 Next, we prove the inverse inequality. To see this, let
$C'=\{I_i\times J_i|
\ 1\leq i\leq t\}$ be a maximal clique of $\Gamma_{reg}(R)$. Setting
$$C_1=\{I_i\times J_i\in C'|\  J_i\ {\rm is \ a \ nontrivial \ ideal \ of }\  T\},$$
we deduce that there are sets $C_2,C_3$ such that $$C'=C_1\cup  (C_2\times \{(0)\})\cup (C_3\times \{T\}).$$
Hence
\begin{eqnarray}\label{1}
|C'|=|C_1|+|C_2|+|C_3|.
\end{eqnarray}
From \cite[Theorem 2.1]{actamathhungar} and Remark \ref{remark}, it follows that $|C_1|\leq 1$; moreover, if $T$ is a field, then $|C_1|=0$.
Now, we follow the proof in the following two cases:

Case 1. Either $C_2=\varnothing$ or $C_3=\varnothing$. If $C_2=\varnothing$ (resp. $C_3=\varnothing$), then by Remark \ref{remark}, $C_3\setminus\{(0)\}$ (resp. $C_2\setminus\{T\}$) is a clique of $\Gamma_{reg}(S)$. This implies that $|C_2|+ |C_3|\leq \omega(\Gamma_{reg}(S))+1$. Thus by (1), we have:
$$\omega(\Gamma_{reg}(R))=|C'|\leq
\begin{cases}
\omega(\Gamma_{reg}(S))+1; & T\, {\rm is \ a \ field}\\
\omega(\Gamma_{reg}(S))+2; & T\, {\rm is \ not \ a \ field}.\\
 \end{cases}
 $$

Case 2. $C_2\neq\varnothing$ and $C_3\neq\varnothing$. In this case, one can easily check that $C_2$ and $C_3$ contain only nontrivial ideals. Also, it follows from Remark \ref{remark} that $C_2\cup C_3$ is a clique of $\Gamma_{reg}(S)$, and this implies that $|C_2\cup C_3|\leq \omega(\Gamma_{reg}(S))$.
We claim that $|C_2\cap C_3|\leq 1$. Suppose to the contrary, $I_1,J_1\in C_2\cap C_3$. Then it is clear that $I_1,J_1$ are nontrivial ideals of $S$, and $\{I_1\times (0),I_1\times T, J_1\times (0), J_1\times T\}\subseteq C'$. Thus  $\overrightarrow{\Gamma_{reg}}(R)$ contains the arcs $I_2\times T \longrightarrow I_1\times (0)$ and $I_1\times T \longrightarrow I_2\times (0)$. Hence $I_1$ contains an $I_2$-regular element and $I_2$ contains an $I_1$-regular element, and this contradicts Remark \ref{remark}(ii). So the claim is proved and hence,
 $$|C_2|+ |C_3|=|C_2\cup C_3|+ |C_2\cap C_3|\leq \omega(\Gamma_{reg}(S)) +1.$$
Thus again by (1), we have: $$\omega(\Gamma_{reg}(R))=|C'|\leq
\begin{cases}
\omega(\Gamma_{reg}(S))+1; & T\, {\rm is \ a \ field}\\
\omega(\Gamma_{reg}(S))+2; & T\, {\rm is \ not \ a \ field}.\\
 \end{cases}
 $$
Therefore, in any case, the assertion follows.}
\end{proof}

For any Artinian ring $R$, by $f(R)$, we denote the number of fields, appeared in the decomposition of $R$ to direct product of local rings.

\begin{prop}\label{omegaformula}
For any Artinian ring $R$,  $\omega(\Gamma_{reg}(R))=2|{\rm Max}(R)|-f(R)-1$.
\end{prop}
\begin{proof}
{If $R$ is a field, then there is nothing to prove. So, assume that $R$ is an Artinian ring which is not a field. Then  \cite[Theorem 8.7]{ati} implies that $R\cong R_1\times R_2\times\cdots\times R_n$, where $n=|{\rm Max}(R)|$ and every $R_i$ is an Artinian local ring. We prove the assertion, by induction on $n$. If $n=1$, then the assertion follows from \cite[Theorem 2.1]{actamathhungar}. Thus we can assume that $n\geq 2$. Now, setting $S=R_1\times R_2\times\cdots\times R_{n-1}$, we follow the proof in the following two cases:

Case 1. $R_n$ is a field. In this case,  the induction hypothesis implies that
$$\omega(\Gamma_{reg}(R'))=2|{\rm Max}(R')|-f(R')-1=2(n-1)-(f(R)-1)-1=2n-f(R)-2;$$
Thus by Lemma \ref{omegaproduct}, we have:
$$\omega(\Gamma_{reg}(R))=\omega(\Gamma_{reg}(R'))+1=2n-f(R)-1.$$

Case 2. $R_n$ is not a field. In this case,  the induction hypothesis implies that
$$\omega(\Gamma_{reg}(R'))=2|{\rm Max}(R')|-f(R')-1=2(n-1)-f(R)-1=2n-f(R)-3;$$
Thus again by Lemma \ref{omegaproduct}, we have:
$$\omega(\Gamma_{reg}(R))=\omega(\Gamma_{reg}(R'))+2=2n-f(R)-1.$$
Therefore, in any case, the assertion follows.
}
\end{proof}

From \cite[Theorem 2.3]{actamathhungar} and Proposition \ref{omegaformula}, we have the following corollary.

\begin{cor}
Let $R$ be an Artinian ring. Then
\begin{enumerate}
\item[\rm (i)] $\omega(\Gamma_{reg}(R))=\chi(\Gamma_{reg}(R))$.
\item[\rm (ii)] If $R$ is reduced, then $\omega(\Gamma_{reg}(R))=|{\rm{Max}}(R)|-1$.
\end{enumerate}
\end{cor}

Now, we state the correct version of Theorem 2.2 from \cite{actamathhungar}.

\begin{thm}\label{artinianomega}
If $R$ is an Artinian ring, then $|{\rm{Max}}(R)|-1\leq\omega(\Gamma_{reg}(R))\leq 2|{\rm{Max}}(R)|-1$. Moreover, $\omega(\Gamma_{reg}(R))=|{\rm{Max}}(R)|-1$ if and only if $R$ is reduced.
\end{thm}

\noindent{\bf\large 4. The Case that $R$ is a Reduced ring
}\\ 

In this section the clique number, the vertex chromatic number and the edge chromatic number of $\Gamma_{reg}(R)$ are determined, when $R$ is a reduced ring.
First, we recall the following interesting
result, due to Eben Matlis.

\begin{prop}\label{matlis} {\rm\cite[Proposition 1.5]{matlis}}
Let $R$ be a ring and $\{{\mathfrak{p}}_1,\ldots,
{\mathfrak{p}}_n\}$ be a finite set of distinct minimal prime ideals
of $R$. Let $S=R\setminus \bigcup_{i=1}^{n}{\mathfrak{p}}_i$. Then
$R_S \cong R_{{\mathfrak{p}}_1}\times\cdots\times
R_{{\mathfrak{p}}_n}$.
\end{prop}


\begin{thm} \label{bII2.11}
Let $R$ be a reduced ring,  $|{\rm Min}(R)|=n\geq 3$ and
$\omega(\Gamma_{reg}(R))<\infty$. Then we have:
\begin{enumerate}
\item[\rm (i)]
$\omega(\Gamma_{reg}(R))=\chi(\Gamma_{reg}(R))=\chi(\Gamma_{reg}(T(R)))=\omega(\Gamma_{reg}(T(R)))=n-1$.
\item[\rm(ii)]  $
\chi'(\Gamma_{reg}(R))=\chi'(\Gamma_{reg}(T(R)))=
\begin{cases}
 2^{n-1}-2;& n\geq 3\\
0;& n=2.
\end{cases}$
\end{enumerate}
\end{thm}

\begin{proof}
{Assume that $\omega(\Gamma_{reg}(R))<\infty$. First we show that every
element of $R$ is an either zero-divisor or unit. By contrary,
suppose that $x\in R$ is neither zero-divisor nor unit. Then it is
not hard to check that $\{(x^n)\}_{n\geq 1}$ is an infinite clique
of $\Gamma_{reg}(R)$, a contradiction. Suppose that
${\rm{Min}}(R)=\{{\mathfrak{p}}_1,\ldots,{\mathfrak{p}}_n\}$, for
some positive integer $n$. If $S=R\setminus
\bigcup_{i=1}^n{\mathfrak{p}}_i$, then \cite [Corollary 2.4]{huc}
implies that $T(R)=R_S$. So by Proposition \ref{matlis}, we have
$T(R)\cong R_{{\mathfrak{p}}_1}\times\cdots\times
R_{{\mathfrak{p}}_n}$. Since $R$ is reduced, by \cite[Proposition
1.1]{matlis}, Part $(1)$, every $R_{{\mathfrak{p}}_i}$ is a field.
 We claim that if $I$ and $J$ are
two distinct vertices of $\overrightarrow{\Gamma_{reg}}(R)$, then $I\longrightarrow J$ is an arc in $\overrightarrow{\Gamma_{reg}}(R)$ if and only if $I_S\longrightarrow J_S$ is an arc in $\overrightarrow{\Gamma_{reg}}(R_S)$. First suppose that $I$ and $J$ are two distinct
non-trivial ideals of $R$ and there is an arc from $I$ to $J$ in
$\overrightarrow{\Gamma_{reg}}(R)$. Since $S$ contains no zero-divisor, we deduce that $I_S$
and $J_S$ are two non-trivial ideals of $R_S$. We show that $I_S\neq J_S$. Suppose to the contrary, $I_S=J_S$. Then
for every $x\in I$, there exists an element $t\in S$
such that $tx\in J$. Since every element in $S$ is a unit,
we deduce that $x\in J$. So $I\subseteq J$. Similarly, one can show that
$J\subseteq I$. Thus $I=J$, a contradiction. Therefore, $I_S\neq
J_S$. Now, let $x\in I$ be a $J$-regular element. Then one can
easily show that $\frac{x}{1}\in I_S$ is a $J_S$-regular element
and so there is an arc from
$I_S$ to $J_S$ in $\overrightarrow{\Gamma_{reg}}(R_S)$. Conversely, let $\frac{x}{s}\in I_S$
be a $J_S$-regular element. Then we show that $x\in I$ is a
$J$-regular element. Suppose to the contrary, $xy=0$, for some $0\neq y\in J$. Then we deduce that $\frac{x}{s}.\frac{y}{1}=0$, a
contradiction. So the claim is proved. Therefore, the graphs $\Gamma_{reg}(R)$ and $\Gamma_{reg}(T(R))$ are isomorphic. Now, since $T(R)$ is the direct product of $n$ fields, (i) follows from Proposition \ref{omegaformula}. Next, we prove (ii). By Theorem \ref{edgechromaticnumber}, we have
$\chi'(\Gamma_{reg}(T(R)))=\Delta(\Gamma_{reg}(T(R)))$. Note that if $n=2$, then $T(R)$ is a direct product of two fields and hence $\Gamma_{reg}(T(R))$ contains no edge. As we saw in the proof of Theorem \ref{edgechromaticnumber}, $\Delta(\Gamma_{reg}(T(R)))=2^{n-1}-2$, for every $n\geq 3$.
Therefore, $\chi'(\Gamma_{reg}(R))=2^{n-1}-2$, and the proof is complete.
}
\end{proof}

The
following corollary is an immediate consequence of Theorem
\ref{bII2.11}.

\begin{cor}\label{alphaomega}
Let $R$ be a reduced ring with finitely many minimal prime ideals such that $\omega(\Gamma_{reg}(R))<\infty$. Then
$$|{\rm{Min}}(R)|=|{\rm Max}(T(R))|=\chi(\Gamma_{reg}(T(R)))+1=\omega(\Gamma_{reg}(R))+1.$$
\end{cor}

Finally, in the remaining of this paper, we see that the finiteness of the clique number and vertex chromatic number of the regular graph of ideals of $R$ depends on those of localizations of $R$ at maximal ideals. Before this, we need to recall the following lemma from \cite{disc.math}.

\begin{lem}{\rm(See \cite[Lemma 9]{disc.math})}\label{1212}
Let $R$ be a ring, $I$ and $J$ be two non-trivial ideals of $R$. If
for every $\mathfrak{m}\in {\rm{Max}}(R)$,
$I_\mathfrak{m}=J_\mathfrak{m}$, then $I=J$.
\end{lem}

\begin{remark}\label{hom}
Let $I$ and $J$ be two distinct non-trivial ideals of $R$ such that $I\longrightarrow J$ is an arc of $\overrightarrow{\Gamma_{reg}}(R)$. Then from \cite[Proposition 1.2.3]{herzog}, we deduce that ${\rm Hom}_R(\frac{R}{I},J)=0$. Moreover, if $R$ is a Noetherian ring, then ${\rm Hom}(\frac{R}{I},J)=0$ implies that $I\longrightarrow J$ is an arc of $\overrightarrow{\Gamma_{reg}}(R)$.
\end{remark}

\begin{thm}
Let $R$ be a Noetherian ring with finitely many maximal ideals. If for every
$\mathfrak{m}\in {\rm{Max}}(R)$,
$\omega(\Gamma_{reg}(R_\mathfrak{m}))$ is finite, then
$\omega(\Gamma_{reg}(R))$ is finite.
\end{thm}

\begin{proof}
{Let
${\rm{Max}}(R)=\{\mathfrak{m}_1,\ldots,\mathfrak{m}_n\}$. Suppose to the contrary, $C=\{J_i\}_{i=1}^{\infty}$ is an infinite clique of
$\Gamma_{reg}(R)$. Then by Remark \ref{hom}, for every $i$ and $j$ with $i\neq j$, either ${\rm Hom}_R(\frac{R}{J_i},J_j)=0$ or ${\rm Hom}_R(\frac{R}{J_j},J_i)=0$. Thus from \cite[Lemma 4.87]{rotman}, we obtain that ${\rm Hom}_{R_{\mathfrak{m}_1}}(\frac{R_{\mathfrak{m}_1}}{(J_i)_{\mathfrak{m}_1}},(J_j)_{\mathfrak{m}_1})=0$ or ${\rm Hom}_{R_{\mathfrak{m}_1}}(\frac{R_{\mathfrak{m}_1}}{(J_j)_{\mathfrak{m}_1}},(J_i)_{\mathfrak{m}_1})=0$, for every $i$ and $j$ with $i\neq j$. Since $\omega(\Gamma_{reg}(R_{\mathfrak{m}_1}))<\infty$, we deduce that
there exists an infinite subset $A_1\subseteq \mathbb{N}$ such that
for every $i,j\in A_1$,
${(J_i)}_{\mathfrak{m}_1}={(J_j)}_{\mathfrak{m}_1}$. Now, using
$\omega(\Gamma_{reg}(R_{\mathfrak{m}_2}))<\infty$, we conclude that
there exists an infinite subset $A_2\subseteq A_1$ such that for
every $i,j\in A_2$,
${(J_i)}_{\mathfrak{m}_2}={(J_j)}_{\mathfrak{m}_2}$. By continuing
this procedure one can see that there exists an infinite subset $A_n
\subseteq A_{n-1}$ such that for every $i,j\in A_n$,
${(J_i)}_{\mathfrak{m}_l}={(J_j)}_{\mathfrak{m}_l}$, for every $l$,
$l=1,\ldots, n$. Therefore, by Lemma \ref{1212}, we get a contradiction. }
\end{proof}

\begin{thm}
Let $R$ be a ring with finitely many maximal ideals. If for every
$\mathfrak{m}\in {\rm{Max}}(R)$, $\chi(\Gamma_{reg}(R_\mathfrak{m}))$
is finite, then $\chi(\Gamma_{reg}(R))$ is finite and moreover,  $$\chi(\Gamma_{reg}(R)\leq\prod_{\mathfrak{m}\in{\rm Max}(R)} (\chi(\Gamma_{reg}(R_{\mathfrak{m}})+2)-2.$$
\end{thm}
\begin{proof}
{Let ${\rm{Max}}(R)=\{\mathfrak{m}_1,\ldots,\mathfrak{m}_n\}$ and
$f_i:\,V(\Gamma_{reg}(R_{\mathfrak{m}_i}))\longrightarrow \{1,\ldots,
\chi(\Gamma_{reg}(R_{\mathfrak{m}_i}))\}$ be a proper vertex coloring
of $\Gamma_{reg}(R_{\mathfrak{m}_i})$, for every $i$, $1 \leq i \leq
n$.  We define a function $f$ on
$\mathbb{I}(R)\setminus\{R\}$ by
$f(I)=(g_1(I_{\mathfrak{m}_1}),\ldots,g_n(I_{\mathfrak{m}_n}))$,
where
$$g_i(I_{\mathfrak{m}_i})=
  \begin{cases}
    0;      & I_{\mathfrak{m}_i}=(0)\,\,\, \\
    -1;       & I_{\mathfrak{m}_i}=R_{\mathfrak{m}_i}\,\,\, \,\,
 \\
    f_i(I_{\mathfrak{m}_i});       & {\rm otherwise}.

\end{cases}$$

\noindent Using Lemma \ref{1212}, it is not hard to check that $f$ is a proper vertex coloring of
$\Gamma_{reg}(R)$ and this completes the proof.}
\end{proof}


{}
~~~~~~~~~~~~~~~~~~~~~~~~~~~~~~~~~~~~~~~~~~~~~~~~~~~~~~~~~~~~~~~~~~~~~~~~~~~~~~~~~~~~~~~~~~~~~~~~~~~~~~~~~~~~~~~~~~~~~~~~~~~~~~~~~~~~~~~~~~~~~~~~~~~~~~~
~~~~~~~~~~~~~~~~~~~~~~~~~~~~~~~~~~~~~~~~~~~~~~~~~~~~~~~~~~~~~~~~~~~~~~~~~~~~~~~~~~~~~~~~~~~~~~~~~~~~~~~~~~~~~~~~~~~~~~~~~~~~~~~~~~~~~~~~~~~~~~~~~~~~~~~~
~~~~~~~~~~~~~~~~~~~~~~~~~~~~~~~~~~~~~~~~~~~~~~~~~~~~~~~~~~~~~~~~~~~~~~~~~~~~~~~~~~~~~~~~~~~~~~~~~~~~~~~~~~~~~~~~~~~~~~~~~~~~~~~~~~~~~~~~~~~~~~~~~~~~~~~~


\begin{thebibliography}{}{\small


\bibitem{disc.math} G. Aalipour,  S. Akbari,   R. Nikandish, M. J. Nikmehr and  F. Shaveisi,  {\it On the coloring
of the annihilating-ideal graph of a commutative ring}, {Discrete
Math.} {\bf 312} (2012), 2620--2626.

\bibitem{rocky} G. Aalipour,  S. Akbari,   R. Nikandish, M. J. Nikmehr and  F. Shaveisi,  {\it Minimal
prime ideals and cycles in annihilating-ideal graphs}, {Rocky Mountain
J. Math.} {\bf 43} (2013),  no. (5), 1415--1425.

\bibitem{khashayarmanesh} M. Afkhami, M. Karimi,  K. Khashayarmanesh, {\it On the regular digraph of ideals of a commutative ring}, { Bull. Aust. Math. Soc.} {\bf 88} (2012), no. 2, 177--189.

\bibitem{atani} Sh. E. Atani, S. D. Pish Hesari and M. Khoramdel, {\it Total graph of a commutative semiring with respect to identity-summand elements}, J. Korean Math. Soc. \textbf{51} (2014), no. 3, 593--607.


\bibitem{ati}  M. F. Atiyah and   I. G. Macdonald,  {\it Introduction to
Commutative Algebra}, Addison-Wesley Publishing Company, 1969.

\bibitem{MBI}  M. Behboodi and  Z. Rakeei, {\it The annihilating ideal
graph of commutative rings I}, {J. Algebra Appl.} {\bf 10} (2011), 727--739.


\bibitem{beineke}  L. W. Beineke and B. J. Wilson,  {\it Selected Topics in Graph Theory}, Academic Press Inc.,
London, 1978.


\bibitem{herzog} W. Bruns,  and  J. Herzog,   {\it Cohen-Macaulay Rings},
Cambridge University Press, 1997.

\bibitem{intersection}  I. Chakrabarty,  S. Ghosh,  T. K. Mukherjee and   M. K. Sen,
{\it Intersection graphs of ideals of rings}, {Discrete Math.} {\bf 309} (2009), 5381--5392.



\bibitem{huc} J. A. Huckaba,  {\it Commutative Rings with Zero-Divisors}, Marcel Dekker Inc., New York, 1988.





\bibitem{matlis}  E. Matlis,  {\it The minimal prime spectrum of a reduced ring},
{Illinois J. Math.} {\bf 27} (1983), no. 3, 353--391.


\bibitem{actamathhungar}   M. J. Nikmehr and F. Shaveisi,   {\it The regular
digraph of ideals of a commutative ring}, {Acta Math. Hungar.} {\bf 134} (2012), 516--528.




\bibitem{redmond} S. P. Redmond,  {\it An ideal-based zero-divisor graph of
a commutative ring}, {Comm. Algebra} {\bf 31} (2003), no. 9, 4425--4443.

\bibitem{rotman} J. J. Rotman,   {\it An Introduction to Homological Algebra}
Springer-Verlag, 2008.

\bibitem{Safaeeyan} S. Safaeeyan, M. Baziar and E. Momtahan, {\it A generalization of the zero-divisor graph for modules}, J. Korean Math. Soc. \textbf{51} (2014), no. 1, 87--98.

\bibitem{F.Sarei} F. E. Kh. Sarei, {\it The total torsion element graph without the zero element of modules over commutative rings}, J. Korean Math. Soc. \textbf{51} (2014), no. 4, 721--734.


\bibitem{west}  D. B. West,  {\it Introduction to Graph Theory}, 2nd ed., Prentice Hall, Upper Saddle River, 2001.


}\end{thebibliography}
\end{document}